\documentclass[10pt]{article}
\usepackage[T1]{fontenc}
\usepackage[utf8]{inputenc}
\usepackage[english]{babel}
\usepackage{amsmath,amssymb,mathrsfs,amsthm}
\usepackage{titlesec}
\usepackage[filecolor=green]{hyperref}
\usepackage{algorithm}
\usepackage{algorithmic}
\usepackage[toc,page]{appendix}
\usepackage{color}
\usepackage[a4paper]{geometry}
\usepackage{tikz}
\usepackage{tikz-cd}

\addtocounter{section}{1}
\addtocounter{subsection}{1}

\newtheorem{theorem}{Theorem}[subsection]
\newtheorem{lemma}[theorem]{Lemma}
\newtheorem{proposition}[theorem]{Proposition}

\theoremstyle{definition}
\newtheorem{definition}[theorem]{Definition}

\newtheorem{example}[theorem]{Example}

\newcommand\eg{\emph{e.g.}}
\newcommand\B{\text{B}}
\newcommand\BB{\emph{B}}
\newcommand\K{\mathbb{K}}
\newcommand\R{\mathbb{R}}
\newcommand\Gr{Gröbner}
\newcommand\KG{\widehat{\K G}}
\newcommand\N{\mathbb{N}}
\newcommand\id[1]{\text{id}_{#1}}
\newcommand\idd[1]{\emph{id}_{#1}}
\newcommand\supp{\text{supp}}
\newcommand\suppp{\emph{supp}}
\newcommand\lm{\text{lm}}
\newcommand\lmm{\emph{lm}}
\newcommand\lc{\text{lc}}
\newcommand\lcc{\emph{lc}}
\newcommand\RO{\textbf{RO}\left(G,\ordred,d\right)}
\newcommand\ROO{\emph{\textbf{RO}}\left(G,\ordredd,d\right)}
\newcommand\ROG{\textbf{RO}}
\newcommand\ROGG{\emph{\textbf{RO}}}
\newcommand\red{\text{red}}
\newcommand\redd{\emph{red}}
\newcommand\nf{\text{nf}}
\newcommand\nff{\emph{nf}}
\newcommand\im{\text{im}}
\newcommand\imm{\emph{im}}
\newcommand\ordred{<_{\text{red}}}
\newcommand\ordredd{<_{\emph{red}}}
\newcommand\obs{\text{obs}(F)}
\newcommand\rewF{\to_F}
\newcommand\G[1]{G^{\left(#1\right)}}
\newcommand\ordop{<^{\text{op}}}
\newcommand\ordopref{\leq^{\text{op}}}
\newcommand\dual[2]{\langle#1\mid#2\rangle}
\newcommand\Span[1]{\langle#1\rangle}
\newcommand\len[1]{\mid#1\mid}
\newcommand\val{\text{val}}
\newcommand\nPol{\K\nf(T)}
\newcommand\nSeries{\widehat{\K\nf(T)}}
\newcommand\nSeriess{\widehat{\K\nff(T)}}
\newcommand\cPol{\K[X]}
\newcommand\cSeries{\K[[X]]}
\newcommand\ncPol{\K\Span{X}}
\newcommand\ncSeries{\K\Span{\Span{X}}}

\newcommand\rewTransSym{\overset{*}{\to}}
\newcommand\rewRef{\leftrightarrow}
\newcommand\rewEquiv{\overset{*}{\leftrightarrow}}

\newcommand\rewTop{\Rightarrow}

\newcommand\discTop[1]{\tau^{\text{d}}_{#1}}

\begin{document}

\title{Topological rewriting systems\\
  applied to standard bases and syntactic algebras}
\author{Cyrille Chenavier\footnote{
      Inria Lille - Nord Europe, équipe Valse, cyrille.chenavier@inria.fr.
  }
  }
\date{}

\maketitle
      
\begin{abstract}
  We introduce topological rewriting systems as a generalisation of
  abstract rewriting systems, where we replace the set of terms by a
  topological space. Abstract rewriting systems correspond to topological
  rewriting systems for the discrete topology. We introduce the
  topological confluence property as an approximation of the confluence
  property. Using a representation of linear topological rewriting
  systems with continuous reduction operators, we show that the
  topological confluence property is characterised by lattice operations.
  Using this characterisation, we show that standard bases induce
  topologically confluent rewriting systems on formal power series.
  Finally, we investigate duality for reduction operators that we relate
  to series representations and syntactic algebras. In particular, we use
  duality for proving that an algebra is syntactic or not.
\end{abstract}
\noindent
\begin{small}\textbf{Keywords:} topological confluence, standard
  bases, series representations and syntactic algebras.\\[0.2cm]
  \textbf{M.S.C 2010 - Primary:} 13F25, 68Q42. \textbf{Secondary:} 03G10.
\end{small}

\tableofcontents

\vspace{1cm}

\begin{center}
  \bf\large 1. INTRODUCTION
  \addcontentsline{toc}{section}{1. Introduction}
\end{center}

\bigskip

Algebraic rewriting systems are computational models used to deduce
algebraic properties through rewriting reasoning. The approach
consists in orienting the generating relations in a presentation by
generators and relations, \eg, of a monoid, a category, a (commutative,
Lie, noncommutative) algebra or an operad, into rewriting rules, and
extend them into rewriting steps. The way we extend these rules takes
into account the underlying algebraic context, but some rewriting
properties have a universal formulation, independent of the context. Such
properties are termination, that is, there is no infinite sequence of
rewriting steps, or confluence, that is, every two rewriting sequences
starting at the same term $t$ may be continued until a common target term
$t'$, as represented on the following diagram:
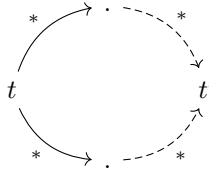
\[\begin{tikzcd}
&
. \arrow[rd, bend left, dashed, "*"]&
\\
t\arrow[ru, bend left, "*"]\arrow[rd, bend right, "*"'] & & t'
\\
&
. \arrow[ru, dashed, bend right, "*"']&
\end{tikzcd}\]
The binary relation $\overset{*}{\to}$ denotes rewriting sequences. Under
hypotheses of termination and confluence, computing irreducible terms,
also called normal forms, has algorithmic applications, for instance to
the decision of the word or the ideal membership problems. It also
provides effective methods for computing linear bases, Hilbert series,
homotopy bases or free resolutions~\cite{MR846601, MR2964639, MR1072284},
and obtain  constructive proofs of coherence theorems, from which we
deduce an explicit description of the action of a monoid on a
category~\cite{MR3347996}, or of homological properties, such as finite 
derivation type, finite homological type \cite{MR3742562, MR920522}, or
Koszulness~\cite{MR265437}.

When one rewrite terms of linear structures, a relation is usually
oriented by rewriting one monomial into the linear combination of other
monomials, and there exist three main approaches for selecting the
rewritten monomial. The most classical one uses monomial orders and
induces a rewriting characterisation of \Gr\ bases: they are generating
sets of polynomials ideals which induce confluent rewriting systems. As a
consequence, effective confluence-based criteria were introduced for
checking if a given set is a \Gr\ basis or one of its numerous
adaptations to different types of algebras or operads
\cite{MR0506423, MR0463136, MR2667136, MR1044911, MR1299371, MR0183753}.
Another approach consists in selecting the
reducible monomials with more flexible orders than monomial ones, which
may be used for  proving Koszulness of algebras for which \Gr\ bases give
no result~\cite{GuiraudHoffbeckMalbos19}. Finally, rewriting steps may be
described in a functional manner~\cite{MR506890, MR1691990, MR3707860},
so that linear rewriting systems are represented by reduction operators.
From this approach, the confluence property is characterised by means of
lattice  operations~\cite{MR3673007}, which provides various applications
to computer algebra and homological algebra: construction of \Gr\ bases
\cite{MR3918052}, computation of syzygies~\cite{MR3850567} or proofs of
Koszulness~\cite{MR1608711, MR1832913, MR3503238, MR3299599}.
\smallskip

Rewriting methods based on monomial orders were also developed for
formal power series, where the leading monomial of a series is the
smallest monomial in its decomposition. Standard bases were introduced by
Hironaka~\cite{MR0199184}, and are analogous to \Gr\ bases: they are
generating sets of power series ideals such that their leading monomials
generate leading monomials of the ideal. A notable difference is that
standard bases are not characterised in terms of the confluence property,
which is an obstruction for allowing more flexible orders than monomial
ones. For instance, for the deglex order induced by $x>y>z$, the
polynomials $z-y,\ z-x,\ y-y^2$ and $x-x^2$ form a standard basis of the
ideal they generate in the power series ring $\K[[x,y,z]]$, but they do
not induce a confluent rewriting system, as illustrated by the following
diagram
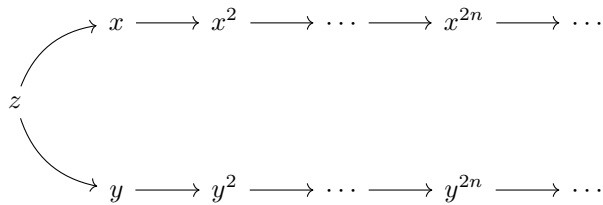
\[\begin{tikzcd}
&
x\arrow[r] & x^2\arrow[r] & \cdots\arrow[r] & x^{2n}\arrow[r] & \cdots
\\
z\arrow[ru, bend left]\arrow[rd, bend right] & & & &
\\
&
y\arrow[r] & y^2\arrow[r] & \cdots\arrow[r] & y^{2n}\arrow[r] & \cdots
\end{tikzcd}\]
However, this diagram becomes confluent when passing to the limit since
the two sequences $(x^{2n})_n$ and $(y^{2n})_n$ both converge to zero for
the $I$-adic topology, where $I$ is the power series ideal generated by
$x,\ y$ and $z$. Note that this asymptotical behaviour of rewriting
sequences is also investigated in computer science, for instance in the
probabilistic $\lambda$-calculus~\cite{faggian2019probabilistic}. 
\smallskip

In the present paper, we introduce a new paradigm of rewriting by
considering rewriting systems on topological spaces, from which we take
into account the topological properties of rewriting sequences. We also
develop the functional approach to rewriting on formal power series. We
get the following two applications: we characterise standard bases in
terms of a topological confluence property and we introduce a criterion
for an algebra to be syntactic.

\paragraph{Topological rewriting systems.} We introduce
\emph{topological rewriting systems}, which, by definition, are triples
$(A,\tau,\to)$, where $\tau$ and $\to$ are a topology and a binary
relation on $A$, respectively. We also introduce the
\emph{topological confluence} property, meaning that two rewriting
sequences starting at the same term $t$ may be continued to reach target
terms in arbitrary neighbourhoods of a term $t'$. Denoting by
$\Rightarrow$ the topological closure of $\overset{*}{\to}$ for the
product of the discrete topology and $\tau$, the topological confluence
property is represented by the following diagram:
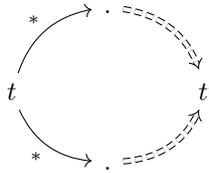
\[\begin{tikzcd}
&
. \arrow[Rightarrow, rd, bend left, dashed]&
\\
t\arrow[ru, bend left, "*"]\arrow[rd, bend right, "*"'] & & t'
\\
&
. \arrow[Rightarrow, ru, dashed, bend right]&
\end{tikzcd}\]
We recover abstract rewriting systems and the usual confluence property
when $\tau$ is the discrete topology. Notice that in this topological
framework, we are not interested in the termination property since we
allow confluence "at the limit". Guided by the aforementioned
applications of reduction operators to computer algebra and homological
algebra, we introduce a topological adaptation of these operators. For
topological vector spaces, monomials form a total family, that is, a free
family generating a dense subspace, and a reduction operator maps such a
monomial into a possibly infinite linear combination of smaller
monomials, that is, a formal series. In
Theorem~\ref{thm:lattice_structure}, we extend the lattice structure
introduced in~\cite{MR3673007} for the discrete topology to topological
vector spaces. From this, we deduce a lattice characterisation of the
topological confluence property in Theorem~\ref{thm:lattice_confluence}. 

\paragraph{Topological confluence for standard bases.} We show that
standard bases are characterised in terms of the topological confluence
property. For that, we first notice that the $I$-adic topology on formal
power series comes from a metric $\delta$, that we recall in Section 4,
and we simply say $\delta$-confluence for the topological confluence
property associated with this metric. Morever, we show in
Proposition~\ref{prop:lattice_car_of_GB} that standard bases are
represented by reduction operators which satisfy the lattice criterion of
topological confluence proven in Theorem~\ref{thm:lattice_confluence}.
Thus, denoting by $\to_R$ the rewriting relation eliminating leading
monomials of a set $R$ of formal power series, our first main result is
stated as follows:
\begin{quote}
  \textbf{Theorem~\ref{thm:standard_bases_confluence}.} \emph{A subset R
    of $\cSeries$ is a standard basis of the ideal it generates if and
    only if the rewriting relation $\to_R$ is $\delta$-confluent.}
\end{quote}

\paragraph{Duality and syntactic algebras.} A formal power series 
uniquely defines a linear form on polynomials. A representation of a
series is a quotient of a polynomial algebra which factorises the linear
form associated with this series, and there always exists a minimal
representation, called the syntactic algebra. An algebra is said to be
syntactic if it is the syntactic algebra of a formal power series. When
they are noncommutative, these algebras may be thought as a
generalisationof automata in the theory of formal languages through the
followinggeneralisation of Kleene's Theorem: a formal power series is
rational if and only if its syntactic algebra is finite-dimensional
\cite{MR593604}. We characterise syntactic algebras in terms of duality
for reduction operators. We expect that this characterisation may be used
for proving that a series is rational or not. Hence, denoting by $\ncPol$
the algebra of noncommutative polynomials over $X$, $\nf(T)$ the set of
normal form monomials for the reduction operator $T$ and by $\nSeries$
the set of formal power series whose nonzero coefficients only involve
elements of $\nf(T)$, our second main result is the following:
\begin{quote}
  \textbf{Theorem~\ref{thm:syntactic_algebra_duality}.} \emph{Let
  $I\subseteq\K\Span{X}$ be an ideal and let T be the reduction operator
    with kernel~I. Then, the algebra $\ncPol/I$ is syntactic if and only if
    there exists $S'\in\nSeriess$ such that I is the greatest ideal
    included in $I\oplus\ker(S')$.}
\end{quote}
This is a duality condition since $\nSeries$ is the kernel of the adjoint
operator of $T$. Finally, we illustrate this criterion with examples of
syntactic and non-syntactic algebras coming from \cite{petitot1992algebre,
  MR593604}.

\paragraph{Organisation.} In Section 2, we introduce topological
reduction operators and show that they admit a lattice structure defined
in terms of kernels. In Section 3.1, we introduce topological confluence
and show that for topological vector spaces, it is characterised in terms
of lattice operations. In Section~3.2, we relate representations of
formal series to duality of reduction operators. In Section 4, we present
two applications of our methods to formal power series. First, we
characterise standard bases in terms of topological confluence. Then, we
formulate a duality criterion for an algebra to be syntactic, and
illustrate it with examples.

\paragraph{Acknowledgment.} The author wishes to thank Michel Petitot for
having suggested to study syntactic algebras, for having pointed out
references on the topic as well as for helpful discussions. The author
also wishes to thank the reviewer for its suggestions to improve the
final version of the paper.
\bigskip

\begin{center}
  \bf\large 2. TOPOLOGICAL REDUCTION OPERATORS
  \addcontentsline{toc}{section}{2. Topological reduction operators}
  \setcounter{section}{2}
\end{center}

\smallskip

In this Section, we introduce reduction operators on topological vector
spaces and show that they admit a lattice structure.

\bigskip

\noindent
{\large\textbf{2.1. Order relation on topological reduction operators}}
\addcontentsline{toc}{subsection}{2.1. Order relation on topological reduction operators}
\setcounter{subsection}{1}
\setcounter{theorem}{0}
\bigskip

\noindent
We fix an ordered set $(G,\ordred)$, a map $d:G\to\R_{>0}$ and a
commutative field $\K$, equipped with the discrete topology. Let $\K G$
be the vector space spanned by $G$. For every $g\in G$, let
\[\pi_g:\K G\to\K,\]
be the linear morphism mapping $v\in\K G$ to the coefficient of $g$ in
$v$. We equip $\K G$ with the metric $\delta$ defined as follows:
\begin{equation}\label{equ:metric}
  \delta(u,v):=\max\left\{d(g)\mid \pi_g(u-v)\neq 0\right\}.
\end{equation}
\noindent
In particular, for every $g\in G$, we have $d(g)=\delta(0,g)$. Let us
denote by $(\KG,\delta)$ the completion of $(\K G,\delta)$. The open ball
in $\KG$ of center $v\in\KG$ and radius $\epsilon>0$ is written
$\B(v,\epsilon)$, and the topological closure of a subset $V$ in $\KG$ is
written $\overline{V}$.

\medskip

Before introducing reduction operators in
Definition~\ref{def:reduction_operators}, we establish some topological
properties of $\KG$, which do not depend on $\ordred$.

\smallskip

\begin{proposition}\label{prop:topological_properties}
  The metric spaces $(\K G,\delta)$ and $(\KG,\delta)$ are topological
  vector spaces. For every $g\in G$, the morphism $\pi_g$ is continuous.
\end{proposition}

\begin{proof}
  For the first part of the Proposition, it is sufficient to show that
  $(\K G,\delta)$ is a topological vector space. Let $\lambda\in\K$,
  $v\in\K G$ and $U$ a neighbourhood of $\lambda v$, so that there exists
  $\epsilon >0$ such that $\delta(\lambda v,u)<\epsilon$ implies
  $u\in U$. Using $\delta(\mu v_1,\mu v_2)\leq\delta(v_1,v_2)$ for every
  $v_1,v_2\in\K G$ and $\mu\in\K$, $\{\lambda\}\times~\B(v,\epsilon)$ is
  an open neighbourhood of $(\lambda,v)$ in the inverse image of $U$
  through the scalar multiplication $(\lambda,v)\mapsto~\lambda v$.
  Hence, the latter is continuous. Let $(v_1,v_2)\in\K G$ and $U$ be a
  neighbourhood of $v_1+v_2$, so that there exists $\epsilon>0$ such that
  $\delta(v_1+v_2,v)<\epsilon$ implies $v\in U$. Using 
  $\delta(u_1+u_2,u'_1+u'_2)\geq\max(\delta(u_1,u'_1), \delta(u_2,u'_2))$,
  $\B(v_1,\epsilon)\times\B (v_2,\epsilon)$ is an open neighbourhood of
  $(v_1,v_2)$ in the inverse image of $U$ through the addition
  $(v_1+v_2)\mapsto v_1+v_2$. Hence, the latter is continuous.

  For the second assertion, it is sufficient to show that $\pi_g$ is
  continuous at $0$, that is, $\pi_g^{-1}(0)$ is open. This is true since
  for $v\in\pi_g^{-1}(0)$, $\B(v,d(g))$ is included in $\pi_g^{-1}(0)$.
  \newline
\end{proof}
\medskip

By density of $\K G$ in $\KG$, the continuous morphism $\pi_g$ induces a
continuous morphism, still written $\pi_g$, from $\KG$ to $\K$.

\smallskip

\begin{definition}
  The \emph{support} of $v\in\KG$ is the set
  $\supp (v):=\left\{g\in G\mid \pi_g(v)\neq0\right\}$.
\end{definition}

Notice that $\pi_g$ being continuous and $\K$ being equipped with the
discrete topology, we have $\pi_g(v)=~\pi_g(u)$ for every $u\in\K G$ such
that $\delta(u,v)$ is small enough.
\medskip

In Formula~\eqref{equ:series}, we describe elements of $\KG$ in terms of
formal series. For that, we need preliminary results on supports that we
present in Lemma~\ref{lem:properties_of_the_support}. In the latter, we
use the following notation: for $v\in\KG$ and $\epsilon>0$, we let
\begin{equation}\label{equ:supp_epsilon}
\supp(v)_{\geq\epsilon}:=\left\{g\in\supp(v)\mid d(g)\geq\epsilon
\right\}.
\end{equation}

\begin{lemma}\label{lem:properties_of_the_support}
  For every $u,v\in\KG$, $\suppp(u+v)$ is included in
  $\suppp(u)\cup\suppp(v)$. Moreover, the sets $\suppp(v)_{\geq\epsilon}$
  are finite and $\suppp(v)$ is countable.
\end{lemma}

\begin{proof}
  Let us show the first assertion. Let $g$ in the complement of
  $\supp(v_1)\cup\supp(v_2)$ in $G$ and let $(u_n)_n$ and $(v_n)_n$ be
  sequences in $\K G$ converging to $u$ and $v$, respectively. For $n$
  large enough, $\pi_g(u_n)$ and $\pi_g(v_n)$ are equal to $0$, so that
  $\pi_g(u+v)=0$.

  Let us show the second assertion. For every $\epsilon>0$, there exists
  $u\in\K G$ such that $\delta(u,v)<\epsilon$, that is, $d(g)<\epsilon$
  for every $g\in\supp(v-u)$. From the first part of the Lemma, 
  $\supp(v)_{\geq\epsilon}$ is included in the finite set $\supp(u)$. 
  Moreover, $\supp(v)$ is the countable union of the finite sets
  $\supp(v)_{\geq 1/n}$, where $n\in\N$, so that it is countable.
  \newline
\end{proof}
\medskip

For every strictly positive integer $n$, let
$\G{n}:=\{g\in G\mid1/n\leq d(g)<1/(n-1)\}$, and for every $v\in\KG$, let
\begin{equation}\label{equ:def_vn}
  v_n:=\sum_{g\in V_n}\pi_g(v)g\in\K G,
\end{equation}
where $V_n:=\supp(v)\cap\G{n}$, that is,
$V_n=\left\{g\in\supp(v)\mid 1/n\leq d(g)<1/(n-1)\right\}$. The sequence
of partial sums $(v_1+\cdots+v_n)_n$ converges to $v$ and the sets $V_n$
form a partition of $\supp(v)$. Hence, we may identify $v$ with the
following formal series:

\begin{equation}\label{equ:series}
  v=\sum_{n\in\N}v_n=\sum_{g\in\supp(v)}\pi_g(v)g.
\end{equation}
\smallskip

In the following Proposition, we present a necessary and sufficient
condition such that $\K G$ is a complete metric space.

\begin{proposition}\label{prop:topology_for_polynomials}
  We have $\KG = \K G$ if and only if $0\in\R$ does not belong to the
  topological closure of $\imm(d)$ in $\R$.
\end{proposition}

\begin{proof}
  We have $0\notin\overline{\im(d)}$ if and only if there exists $n>0$
  such that $d(g)\geq1/n$, for every $g\in G$. With the notations of
  \eqref{equ:series}, that means that each $v\in\KG$ is equal to
  $v_1+\cdots+v_n$, that is, $v\in\K G$.
  \newline
\end{proof}
\medskip

\begin{example}\label{ex:pol_and_formal_series}
  Let $X$ be a finite set of indeterminates and let $G$ be the set of
  (non)commutative monomials over $X$, so that $\K G$ is isomorphic to
  the free (non)commutative polynomial algebra over $X$. If $d$ is
  constant equal to $1$, then the metric $\delta$  of \eqref{equ:metric}
  satisfies $\delta(f,g)=1$ whenever $f\neq g$, so that $\KG$ is the
  (non)commutative polynomial algebra equipped with the discrete
  topology. If $d(m)=1/2^n$ for every (non)commutative monomial of degree
  $n$, then $\KG$ is the set of (non)commutative formal power series over
  $X$ and $\delta$ is the metric of the $I$-adic topology, where $I$ is
  the two-sided power series ideal generated by $X$. In Section 4, we
  investigate this topological vector space and rewriting systems on it
  in more details.
\end{example}

Now, we introduce reduction operators and present some of their basic
properties.
\smallskip

\begin{definition}\label{def:reduction_operators}
  A \emph{reduction operator} relative to $\left(G,\ordred,d\right)$ is a
  continuous linear projector of $\KG$ such that for every $g\in G$,
  $T(g)\neq g$ implies $g'\ordred g$, for every $g'\in\supp(T(g))$.
\end{definition}

\medskip

The set of reduction operators is written $\RO$ and for $T\in\RO$, we
let
\[
\nf(T):=\left\{g\in G\mid T(g) = g\right\}\ \ \text{and}\ \
\red(T):=\left\{g\in G\mid T(g)\neq g\right\}.
\]
The element $g\in G$ is called a \emph{T-normal form} if $g\in\nf(T)$ or
\emph{T-reducible} if $g\in\red(T)$.

\smallskip

\begin{proposition}\label{prop:ker_im_closed}
  Let $T\in\ROO$. The subspaces $\imm(T)$ and $\ker(T)$ are the closed
  subspaces spanned by $\nff(T)$ and
  $\left\{g-T(g)\mid g\in\redd(T)\right\}$, respectively.
\end{proposition}

\begin{proof}
  The set $\nf(T)$ is included in $\im(T)$ and the latter is closed since
  it is the inverse image of $\{0\}$ by the continuous map $\id{V}-T$.
  Hence, the closure $\overline{\K\nf(T)}$ of $\K\nf(T)$ is included in
  $\im(T)$. In the same manner, we show that 
  $\overline{\K\left\{g-T(g)\mid g\in\red(T)\right\}}\subseteq\ker(T)$.
  
  Let us show the converse inclusions. Let $v\in\KG$ and $(v_n)_n$ be a
  sequence in $\K G$ converging to $v$. For every $n\in\N$, there exist
  $u_n\in\K\nf(T)$ and $w_n\in\K\red(T)$ such that $v_n=u_n+w_n$. By
  continuity and linearity of $T$, $T(v)$ is the limit of the sequence
  $(z_n)_n$, where $z_n:=u_n+T(w_n)$. The latter belongs to $\K\nf(T)$.
  If $v\in\im(T)$, we have $v=T(v)$, so that $v=\lim(z_n)$ belongs to
  $\overline{\K\nf(T)}$. If $v\in\ker(T)$, we have $v=v-T(v)$, so that
  $v=\lim(w_n-T(w_n))$ belongs to
  $\overline{\K\left\{g-T(g)\mid g\in\red(T)\right\}}$.
  \newline
\end{proof}
\medskip

The following Lemma is used to prove Proposition~\ref{prop:order_on_RO},
where we prove that reduction operators form a poset.

\smallskip

\begin{lemma}\label{lem:int_lemma_for_order}
  Let $T,\ T'\in\ROO$, $g\in G$ and $v\in\KG$.
  \begin{enumerate}
  \item\label{it:supp_of_RO} If for every $g'\in\suppp(v)$, we have
    $g'<g$, then $g$ does not belong to $\suppp(T(v))$.
  \item\label{it:nf_is_increasing}
    If $\ker(T)\subseteq\ker(T')$, then $\nff(T')\subseteq\nff(T)$.
  \end{enumerate}
  \end{lemma}

\begin{proof}
  First, we show Point~\ref{it:supp_of_RO}. If $v\in\K G$, the result is
  a consequence of Lemma~\ref{lem:properties_of_the_support}. If $v$ does
  no belong to $\K G$, let $v_n$ as in \eqref{equ:def_vn}, so that $v$ is
  the limit of $(w_n)_n$, where $w_n:=v_1+\cdots+v_n$. Each
  $g'\in\supp(w_n)$ is strictly smaller than $g$, so that
  $g\notin\supp(T(w_n))$. By continuity, $(T(w_n))_n$ converges to
  $T(v)$, so that $g\notin\supp(T(v))$.
  
  Let us show Point~\ref{it:nf_is_increasing} by contrapositive. Let us
  assume that there exists $g\in\nf(T')$ such that $g\notin\nf(T)$. The
  vector $v:=g-T(g)$ belongs to $\ker(T)$ but does not belong to
  $\ker(T')$: otherwise, we would have $g=T'(T(g))$, which is not
  possible from Point~\ref{it:supp_of_RO}.
  \newline
\end{proof}
\medskip

\begin{proposition}\label{prop:order_on_RO}
  The binary relation $\preceq$ on $\ROO$ defined by $T\preceq T'$ if
  $\ker(T')\subseteq\ker(T)$ is an order relation.
\end{proposition}

\begin{proof}
  The relation $\preceq$ is reflexive and transitive. Moreover,
  $\ker(T)=\ker(T')$ implies $\nf(T)=\nf(T')$ from
  Point~\ref{it:nf_is_increasing} of Lemma~\ref{lem:int_lemma_for_order}.
  From Proposition~\ref{prop:ker_im_closed}, $\nf(T)=\nf(T')$ implies
  $\im(T)=\im(T')$. Hence, $T$ and $T'$ are two projectors with same
  kernels and images, so that they are equal and $\preceq$ is
  antisymmetric.
  \newline
\end{proof}
\medskip

\noindent
{\large\textbf{2.2. Elimination maps}}
\addcontentsline{toc}{subsection}{2.2. Elimination maps}
\setcounter{subsection}{2}
\setcounter{theorem}{0}
\bigskip

\noindent
From Propositions~\ref{prop:ker_im_closed} and \ref{prop:order_on_RO},
the kernel map induces an injection of $\RO$ into closed subspaces of
$\KG$. In this Section, we introduce a sufficient condition such that
this injection is surjective. Moreover, we deduce a lattice structure on
reduction operators in Theorem~\ref{thm:lattice_structure}. 
\medskip

Throughout the Section, we assume that $\ordred$ is a total order and $d$
is an elimination map, where this notion is introduced in the following
definition. Before, recall that $\G{n}$ denotes elements $g\in G$ such
that $1/n\leq d(g)<1/(n-1)$.
\smallskip

\begin{definition}\label{def:elimination_map}
  We say that $d$ is an \emph{elimination map} with respect to
  $(G,\ordred)$ if it is non-decreasing and if the sets $\G{n}$ equipped
  with the order induced by $\ordred$ are well-ordered sets.
\end{definition}
\medskip

Under these hypotheses, we may introduce the notion of leading monomial.
The order $\ordred$ being total, for every $v\in\K G$, there exists a
greatest element in the support of $v$, written $\max(\supp(v))$. For
$v\in\KG$, we let $v=\sum v_n$ as in \eqref{equ:series} and we denote by
$n_0$ the smallest $n$ such that $v_n\neq 0$. The map $d$ being
non-decreasing, $\max(\supp(v_{n_0}))$ is the greatest element of
$\supp(v)$. We let
\smallskip
\[
\lm(v):=\max(\supp(v_{n_0}))\ \ \text{and}\ \ \lc(v)=\pi_{\lm(v)}(v).
\]
The elements $\lm(v)\in G$ and $\lc(v)\in\K$ are respectively called the
\emph{leading monomial} and the \emph{leading coefficient} of $v$.
\medskip

The construction of the inverse of $\ker$ is based on a property of
leading monomials of closed subspaces presented in
Proposition~\ref{prop:construction_elimination_vectors}. In order to show
the latter, we need the following intermediate lemma.

\begin{lemma}\label{lem:truncated_reduced_family}
  Let $V$ be a subspace of $\KG$ and let $n$ be an integer. There exists
  a family $(v_g)_{g\in\G{n}}\subset V$ such that the following hold:
  \begin{enumerate}
  \item[$\bullet$] $v_g=0$ if and only if there is no $v\in V$ such that
    $\lmm(v)=g$;
  \item[$\bullet$] if $v_g\neq 0$, then $\lmm(v_g)=g$, $\lcc(v_g)=1$ and
    for $g'\neq g$ in $\G{n}$, $g\notin\suppp(v_{g'})$.
  \end{enumerate}
\end{lemma}

\begin{proof}
  We proceed by induction along the well-founded order $\ordred$: assume
  that for every $g'\ordred g$ in $\G{n}$, $v_{g'}$ is constructed. If
  there is no $v\in V$ such that $\lm(v)=g$, we let $v_g=0$. Else, we
  choose such a $v$ with $\lc(v)=1$ and we let
  $v_g=v-\sum\pi_{g'}(v)v_{g'}$, where the sum is taken over $g'\in\G{n}$
  such that $g'\ordred g$.
  \newline
\end{proof}
\medskip

In the statement of
Proposition~\ref{prop:construction_elimination_vectors}, we use the
notion of a \emph{total basis} of a topological vector space $V$, that
is, a free family which spans a dense subspace of $V$.
\smallskip

\begin{proposition}\label{prop:construction_elimination_vectors}
  Let $V$ be a subspace of $\KG$. There exists a family $(v_g)_{g\in G}$
  of $\overline{V}$ such that the following hold:
  \begin{itemize}
  \item[$\bullet$] $v_g=0$ if and only if there is no $v\in V$
    such that $\lmm(v)=g$;

    \smallskip

  \item[$\bullet$] if $v_g\neq 0$, then $\lmm(v_g)=g$, $\lcc(v_g)=1$ and
    $g\notin\suppp(v_{g'})$ for every $g'\neq g$.
  \end{itemize}
  \smallskip
  In particular, nonzero elements of this family form a total basis of
  $\overline{V}$.
\end{proposition}

\begin{proof}
  For every $g\in G$, let us define by induction the sequence $(v^n_g)_n$
  as follows: $v^n_g=0$ if $d(g)<1/n$, $v^n_g$ is chosen as in
  Lemma~\ref{lem:truncated_reduced_family} if $g\in\G{n}$, that is, if
  $1/n\leq d(g)<1/(n-1)$, and
  \[v^n_g=v^{n-1}_g-\sum_{g'\in\G{n}}\pi_{g'}(v^{n-1}_g)v^n_{g'},\]
  if $d(g)\geq 1/(n-1)$. For every $n,\ m$ and $g$, $\lm(v^n_g-v^m_g)$
  belongs to $\G{k}$, for $k\geq\min(n,m)$. Hence, we have
  $\delta(v^n_g,v^m_g)<1/\min(n,m)$, so that $(v^n_g)_n$ is Cauchy and
  converges to $v_g\in\overline{V}$. By construction, $(v_g)_g$ satisfies
  the first two properties stated in the Proposition. Let us show that
  its nonzero elements form a total basis. They form a free family since
  the leading monomials of its elements are pairwise distinct. Moreover,
  this family is dense since for every $v\in\overline{V}$ and every
  $\epsilon >0$, the element
  \[v_{\epsilon}:=\sum_{g\in\supp(v)_{\geq\epsilon}}\pi_g(v)v_g,\]
  $\supp(v)_{\geq\epsilon}$ being defined as in \eqref{equ:supp_epsilon},
  belongs to $\B(v,\epsilon)$. Indeed, assume by contradiction that there
  exists $g'\in\supp(v-v_{\epsilon})$ such that $d(g')\geq\epsilon$. 
  There exists $g\in\supp(v)_{\geq\epsilon}$ such that $g'\in\supp(v_g)$
  but $g'\neq g$. The maximal $g'$ for $\ordred$ is the leading monomial
  of $v-v_{\epsilon}\in V$, so that $v_{g'}\neq 0$. This is a
  contradiction since this condition implies $g'\notin\supp(v_g)$.
  \newline
\end{proof}
\medskip

With the notations of the previous Proposition, the elements $v_g$ are
limits of sequences $(v^n_g)_n$ of elements of $V$ such that
$\lm(v^n_g)=g$. The family of nonzero $v_g$'s being total, for every
$v\in\overline{V}$, $\lm(v)=g$ for some $v_g$, so that it belongs to
$\lm(V):=\{\lm(u)\mid u\in V\}$. Hence, the following formula holds:

\begin{equation}\label{equ:leading_monomials}
  \lm(V)=\lm(\overline{V}).
\end{equation}
\smallskip

We can now introduce the main result of this Section.
\smallskip

\begin{theorem}\label{thm:lattice_structure}
  Assume that $\ordredd$ is a total order and that $d$ is an elimination
  map with respect to $(G,\ordredd)$. The kernel map induces a bijection
  between $\ROO$ and closed subspaces of $\KG$. In particular, $\ROO$
  admits lattice operations $\left(\preceq,\wedge,\vee\right)$, defined
  by
  \smallskip
  \begin{itemize}
  \item[$\bullet$] $T_1\preceq T_2$ if $\ker(T_2)\subseteq\ker(T_1)$;
    \smallskip
  \item[$\bullet$] $T_1\wedge T_2:= \ker^{-1}\left(\overline{\ker(T_1)+
  \ker(T_2)}\right)$;
    \smallskip
  \item[$\bullet$] $T_1\vee T_2:=\ker^{-1}\left(\ker(T_1)\cap\ker(T_2)
    \right)$.
  \end{itemize}
\end{theorem}

\begin{proof}
  It is sufficient to show that $\ker$ is surjective. Consider a closed
  subspace $V$ of $\KG$ and let $(v_g)_g$ as in
  Proposition~\ref{prop:construction_elimination_vectors}. Consider the
  linear map $T:\KG\to\KG$ defined by $T(g)=v_g-g$ if $v_g\neq 0$ and
  $T(g)=g$, otherwise. This map defines a linear projector, compatible
  with $\ordred$ and is continuous at $0$ since for every $v\in\KG$, we
  have $\delta(T(v),0)\leq\delta(v,0)$. Hence, $T$ is continuous, so that
  it is a reduction operator. Moreover, the set of nonzero $v_g$'s is a
  total basis of $\ker(T)$ and $V$ from
  Propositions~\ref{prop:ker_im_closed}
  and~\ref{prop:construction_elimination_vectors}, so that we have
  $\ker(T)=V$.
  \newline
\end{proof}
\medskip

\begin{example}\label{ex:polynomials}
  We assume that $0$ does not belong to the closure of $\im(d)$, for
  instance, $d$ is constant equal to $1$. We fix a well-order $<$ on $G$
  and we choose $\ordred$ equal to $<$, so that $d$ is an elimination
  map. Moreover, the topology induced by $\delta$ is the discrete
  topology, so that every subspace is closed. Hence, we recover a result
  from ~\cite{MR3673007}: $\ker$ induces a bijection between subspaces of
  $\K G$ and reduction operators. The interpretations of the lattice
  operations on $\RO$ were given in~\cite[Proposition 2.3.6]{MR3673007}
  and~\cite[Proposition 2.1.4]{MR3850567}: $T_1\wedge T_2$ characterises
  the equivalence relation on $\K G$ induced by the binary relation
  $v\to T_i(v)$, $v\in\K G$ and $i=1,2$, and $\ker(T_1\vee T_2)$ is
  isomorphic to the space of syzygies for $T_1$ and~$T_2$.
\end{example}

\begin{example}\label{ex:formal_series}
  Assume that $G$ is infinite and equipped with a well-order $<$ such that
  $d$ is strictly decreasing. In particular, $0$ belongs to the closure of
  $\im(d)$. For $\ordred$, we choose the opposite order of $<$, that is,
  $g\ordred g'$, whenever $g'<g$. Hence, $d$ is an elimination map, $\KG$
  is the set of formal series as in~\eqref{equ:series}, where the sum
  may be infinite, and $\ker$ induces a bijection between closed subspaces
  of formal series and reduction operators. We point out that there exist
  subspaces which are not closed, as illustrated in the following
  Proposition.
\end{example}
\smallskip

\begin{proposition}\label{prop:non_closed_dense_subspace_formal_series}
  Let $G:=\{g_1<g_2<\cdots\}$ be a countable well-ordered set and
  $d:G\to\R_{>0}$, defined by $d(g_n)=1/n$. The subspace
  $V:=\K\{g_n-g_{n+1}\mid n\geq 1\}\subset\KG$ is dense and different
  from $\KG$.
\end{proposition}
\smallskip

\begin{proof}
  For every $n\geq 1$, $g_n$ does not belong to $V$ but is equal to
  $\lim(g_n-g_k)_k$.
  \newline
\end{proof}
\bigskip

\begin{center}
  \bf\large 3. CONFLUENCE AND DUALITY
  \addcontentsline{toc}{section}{3. Confluence and duality}
  \setcounter{section}{3}
\end{center}

\smallskip

In this Section, we investigate the rewriting and duality properties of
reduction operators.

\bigskip

\noindent
{\large\textbf{3.1. Topological confluence}}
\addcontentsline{toc}{subsection}{3.1. Topological confluence}
\setcounter{subsection}{1}
\setcounter{theorem}{0}
\bigskip

Throughout this Section, we fix a set $F\subseteq\RO$, where $d$ is an
elimination map with respect to $(G,\ordred)$. Our objective is to
introduce a confluence-like property for $F$. First, we recall
from~\cite{MR1629216} classical notions of rewriting theory and introduce
a topological adaptation of rewriting systems.
\medskip

An \emph{abstract rewriting system} is a pair
$(A,\to)$, where $A$ is a set and $\to$ is a binary relation on $A$. We
write $a\to b$ instead of $(a,b)\in\to$. We denote by $\rewRef$ and
$\rewTransSym$ the symmetric and the reflexive transitive closures of
$\to$, respectively. Hence, $\rewEquiv$ is the reflexive transitive
symmetric closure of $\to$. If $a\rewTransSym b$, we say that $a$
\emph{rewrites} into $b$. We say that $a$ and $b$ are \emph{joinable} if
there exists $c$ such that both $a$ and $b$ rewrite into $c$. We say that
$\to$ is \emph{confluent} if whenever $a$ rewrites into $b$ and $c$, then
$b$ and $c$ are joinable.
\smallskip

A \emph{topological rewriting system} is a triple $(A,\tau,\to)$, where
$(A,\tau)$ is a topological space and $\to$ is a rewriting relation on
$A$. We denote by $\rewTop$ the topological closure of $\to$ for the
product topology $\discTop{A}\times\tau$, where $\discTop{A}$ is the
discrete topology on $A$. In other words, we have $a\rewTop b$ if and
only if every neighbourhood $V$ of $b$ contains $b'$ such that $a$
rewrites into $b'$. By this approach, we formulate the following
topological notion of confluence.
\smallskip

\begin{definition}\label{def:topo_conf}
  The rewriting relation $\to$ is $\tau$-\emph{confluent} if whenever $a$
  rewrites into $b$ and $c$, then there exists $d$ such that $b\rewTop d$
  and $c\rewTop d$.
\end{definition}
\noindent
Notice that if $\tau=\discTop{A}$, then $\tau$-confluence is equivalent
to confluence.
\medskip

We equip $\KG$ with the topology induced by $\delta$ and we associate to
$F$ the topological rewriting system $(\KG,\delta,\rewF)$ defined by
\[v\rewF T(v),\ \forall v\in\KG,\ T\in F.\]
Moreover, we consider the operator $\wedge F\in\RO$ and the set
$\nf(F)\subseteq G$ defined as follows:
\[\wedge F:=\ker^{-1}\left(\overline{\sum_{T\in F}\ker(T)}\right),\
\nf(F):=\bigcap_{T\in F}\nf(T).\]
From Point~\ref{it:nf_is_increasing} of
Lemma~\ref{lem:int_lemma_for_order}, $\nf(\wedge F)$ is included in
$\nf(F)$, and we let
\[
\obs:=\nf(F)\setminus\nf(\wedge F).
\]
The elements of $\obs$ are called the \emph{obstructions} of $F$.
\smallskip

\begin{definition}\label{def:confluence_for_RO}
  A set $F\subseteq\textbf{RO}(G,\ \ordred,\ d)$ is said to be
  \emph{confluent} if $\obs=\emptyset$.
\end{definition}
\medskip

In Theorem~\ref{thm:lattice_confluence}, we show that $F$ is confluent if
and only if it induces a $\delta$-confluent rewriting relation. For that,
we need the following intermediate definition and results.
\smallskip

\begin{definition}
  Let $v\in\ker(\wedge F)$. A decomposition
  $v=\sum_{i=0}^n\lambda_i(g_i-T_i(g_i))+r$, where $\lambda_i\neq 0$,
  $T_i\in F$ and $r\in\ker(\wedge F)$, of $v$ is said to be
  \emph{admissible} if $g_i\leq\lm(v)$, for every $1\leq i\leq n$, and
   $\lm(r)<\lm(v)$.
\end{definition}

The following lemma establishes the link between the confluence property
and reduction of $S$-polynomials into zero in our framework.
\smallskip

\begin{lemma}\label{lem:S-pol}
  If $\rewF$ is $\delta$-confluent, then for $g\in G$ and $T,T'\in F$,
  $(T-T')(g)\rewTop_F0$ and $(T-T')(g)$ admits an admissible
  decomposition.
\end{lemma}

\begin{proof}
  Let us show the first part of the lemma. Assume by induction that for
  every $n\geq 0$, a finite composition $R_n$ of elements of $F$ has been
  constructed such that $\delta(R_n(T-T')(g),0)< 1/n$. By induction on
  the well-ordered set $\G{n+1}$, there exists a finite composition $R$
  of elements of $F$ such that
  \begin{equation}\label{equ:supp_condition}
    \supp(R\circ R_n(T(g))_{\geq 1/(n+1)}\cap
    \supp(R\circ R_n(T'(g))_{\geq 1/(n+1)}\subseteq\nf(F).
  \end{equation}
  Letting $R_{n+1}:=R\circ R_n$, $g$ rewrites into $v:=R_{n+1}(T(g))$ and
  $v':=R_{n+1}(T'(g))$. By $\delta$-confluence, there exists $w$ such
  that $v\rewTop w$ and $v'\rewTop w$, and from
  \eqref{equ:supp_condition},  each $g'\in G$ with $d(g')\geq n+1$ of
  $\supp(v)$ and $\supp(v')$ are normal forms. Hence,
  $\pi_{g'}(v)=\pi_{g'}(v')$, so that $\delta(R_n(T-T')(g),0)< 1/(n+1)$.
  Hence, $(T-T')(g)\rewTop_F0$.

  The second part of the lemma is a consequence of the the first one and
  the following fact: if $v\rewTop_F0$, then $v$ admits an admissible
  decomposition. Indeed, for every $n$, we let $v=(v-R_n(v))+R_n(v)$,
  where $R_n$ is a finite composition of elements of $F$ such that
  $\delta(R_n(v),0)<1/n$. Then, $v-R_n(v)$ is a linear combination of
  elements of the form $g-T(g)$, with $g\leq\lm(v)$. Moreover, for $n$
  large enough, we have $\lm(R_n(v))<\lm(v)$, so that $v$ admits an
  admissible decomposition.
\end{proof}

\begin{proposition}\label{prop:admissible_decompo}
  If $\rewF$ is $\delta$-confluent, then every $v\in\ker(\wedge F)$
  admits an admissible decomposition.
\end{proposition}

\begin{proof}
  We adapt the proof of \cite[Lemma 4.2]{MR1075338} to our situation.

  By density of $\sum_{T\in F}\ker(T)$ in $\ker(\wedge F)$, $v$ admits a
  decomposition $v=\sum_{i\geq 1}^n\lambda_i(g_i-T_i(g_i))+r$, such that
  $\lm(r)<\lm(v)$. Without lose of generality, assume that $(g_i)_i$ is
  not increasing. If $g_1\leq\lm(v)$, then the chosen decomposition is
  admissible. Otherwise, we show by induction on the greatest $k\geq 2$
  such that $g_k=g_1$, that $v$ admits another decomposition
  \begin{equation}\label{equ:better_decompot}
    \sum_{i=1}^{n'}\lambda'_i(g'_i-T'_i(g'_i))+r',\ g'_i\geq g'_{i+1},\
    g'_1<g_1,\ \lm(r')<\lm(v).
  \end{equation}
  If $k=2$, we have $\lambda_2=-\lambda_1$, so that
  $v=\lambda_1(T_2-T_1)(g_1)+\sum_{i=3}^n\lambda_i(g_i-T_i(g_i))+r$. From
  Lemma~\ref{lem:S-pol}, $(T_1-T_2)(g_1)$ admits an admissible
  decomposition $\sum_{i=1}^{n'}\lambda'_i(g'_i-T'_i(g'_i))+r'$. In
  particular, we have $g'_1<g_1$, so that
  $\sum_{i=1}^{n'}\lambda'_i(g'_i-T'_i(g'_i))+\Big(\sum_{i=3}^n\lambda_i
  (g_i-T_i(g_i))+r+r'\Big)$ is a decomposition of  $v$ such as in
  \eqref{equ:better_decompot}. If $k\geq 3$, we write
  \[v=\lambda_1(T_2-T_1)(g_1)+\left((\lambda_1+\lambda_2)(g_1-T_2(g_1))+
  \sum_{i=3}^k\lambda_i(g_1-T_i(g_1))+
  \sum_{i=k+1}^n\lambda_i(g_i-T_i(g_i))+r\right).\]
  The first term of the sum admits an admissible decomposition from
  Lemma~\ref{lem:S-pol}. Moreover, the element
  $v':=v-\lambda_1(T_2-T_1)(g_1)$ admits a decomposition such as in
  \eqref{equ:better_decompot} by induction hypothesis. The sum of these
  two decomposition gives a decomposition of $v$ such as in
  \eqref{equ:better_decompot}. By applying inductively the decomposition
  \eqref{equ:better_decompot} to $v$, we deduce that it admits an
  admissible decomposition.
\end{proof}

\begin{proposition}\label{prop:intermediate_results_for_confluence}
  Let $v,v'\in\KG$.
  \begin{enumerate}
  \item\label{it:local_equiv} $v-v'\in\ker(\wedge F)$ if and only if for
    every $\epsilon>0$, there exists
    $v_{\epsilon}\in\BB(0,\epsilon)$ such that
    $v\rewEquiv_Fv'+v_{\epsilon}$.
  \item\label{it:CR_lattice} If $F$ is confluent, then
      $v\rewTop_F(\wedge F)(v)$.
  \end{enumerate}
\end{proposition}

\begin{proof}
  Let us show Point~\ref{it:local_equiv}. First, assume that
  $v_{\epsilon}\in\B(0,\epsilon)$ exists such that
  $v\rewEquiv_Fv'+v_{\epsilon}$. By definition and by continuity of
  $\wedge F$, $v-v'-v_{\epsilon}$ belongs to its kernel and
  $(\wedge F)(v_{\epsilon})$ goes to zero when $\epsilon$ does.
  Hence, $v-v'\in\ker(\wedge F)$. Conversely, $v-v'\in\ker(\wedge F)$
  implies that $v=\sum_{i\geq 1}\lambda_i(g_i-T_i(g_i))+v'$, where the
  sum maybe infinite, $T_i\in F$ and the sequence $(d(g_i))_i$ goes to
  $0$. For every $k\geq 1$, letting
  $v_k:=\sum_{i\geq k}\lambda_i(g_i-T_i(g_i))$, we have
  $T_k(v_k+v')=T_k(v_{k+1}+v')$, so that $v\rewEquiv_Fv_k+v'$. The
  sequence $(v_k)_k$ goes to zero, which shows the direct implication of
  Point~\ref{it:local_equiv}.

  Let us show Point~\ref{it:CR_lattice}. By induction, assume that a
  finite composition $R_n$ of elements of $F$ has been constructed such
  that $\supp(R_n(v))\cap\G{k}$ is included in $\nf(\wedge F)$, for every
  $k<n$. By confluence of $F$, $\red(\wedge F)$ is the union of the sets
  $\red(T)$, where $T\in F$. By induction on the well-ordered set
  $\G{n}$, we construct a finite composition $R$ of elements of $F$ such
  that $\supp(R(R_n(v)))\cap\G{n}$ is included in $\nf(\wedge F)$. We let
  $R_{n+1}:=R\circ R_n$. By this iterative construction, we get a
  sequence $(v_n)_n$, where $v_n:=R_n(v)$, such that $v\rewTransSym_Fv_n$
  and $\delta((\wedge F)(v),v_n)\leq 1/n$. Passing to the limit, we get
  $v\rewTop_F(\wedge F)(v)$.
\end{proof}

\begin{theorem}\label{thm:lattice_confluence}
  Let $F\subseteq\emph{\textbf{RO}}(G,\ <_{\emph{red}},\ d)$. Then, $F$
  is confluent if and only if $\rewF$ is $\delta$-confluent.
\end{theorem}

\begin{proof}
  Assume that $F$ is confluent and let $v\in\KG$ which rewrites both into
  $v_1$ and $v_2$. From~\ref{it:CR_lattice} of
  Proposition~\ref{prop:intermediate_results_for_confluence},
  $v_i\rewTop_F(\wedge F)(v_i)$, for $i=1,2$. We also have
  $v_1\rewEquiv_Fv_2$, so that $(\wedge F)(v_1)=(\wedge F)(v_2)$
  from~\ref{it:local_equiv} of
  Proposition~\ref{prop:intermediate_results_for_confluence}. Hence,
  $\rewF$ is $\delta$-confluent.

  If $\rewF$ is $\delta$-confluent, from
  Proposition~\ref{prop:admissible_decompo}, every $v\in\ker(\wedge F)$
  admits an admissible decomposition $\sum\lambda_i(g_i-T_i(g_i))$ such
  that $\lm(v)$ is the greatest of the $g_i$'s. Moreover, the $g_i$'s may
  be chosen in such a way that they belong to $\red(T_i)$, so that
  $\lm(v)$ is $T$-reducible for $T\in F$. Hence, $\red(\wedge F)$ being
  equal to  $\bigcup\{\lm(v)\mid v\in\ker(\wedge F)\}$, it is the union
  of the sets $\red(T)$, $T\in F$, so that $F$ is confluent.
\end{proof}

\bigskip

\noindent
{\large\textbf{3.2. Duality and series representations}}
\addcontentsline{toc}{subsection}{3.2. Duality and series representations}
\setcounter{subsection}{2}
\setcounter{theorem}{0}
\bigskip

In this Section, we fix a countable set $G$, equipped with a well-order
$<$. We fix a strictly decreasing map $d:G\to\R {>0}$. Considering
$\ordop$, the opposite order of $<$, $\KG$ is the set of formal series
over $G$ and $d$ is an elimination map, as pointed out in Example
\ref{ex:formal_series}. We consider the following two sets of reduction
operators: $\ROG(G,<,1)$ and $\ROG(G,\ordop,d)$, where $1$ is the
function $g\mapsto 1$. For simplicity, we say reduction operators on
$\K G$ and $\KG$, and we denote these sets by $\ROG(G,<)$ and
$\ROG(\widehat G,\ordop)$, respectively. 
\medskip

We denote by $\K G^*$ the algebraic dual of $\K G$. For
$\varphi\in\K G^*$ and $v\in\K G$, let $\dual{\varphi}{v}\in\K$ by the
result obtained by applying $\varphi$ to $v$. In the sequel, we identify
$\K G^*$ to $\KG$ through the isomorphism
$\K G^*\to\KG,\ \varphi\mapsto\sum\dual{\varphi}{g}g$. For $T$ an
endomorphism of $\K G$, we denote by $T^*:\KG\to\KG$ the adjoint operator
defined by $T^*(\varphi)=\varphi\circ T$. For a subspace
$V\subseteq\K G$, we denote by $V^{\bot}\subseteq\KG$ the orthogonal
space of $V$, that is, the set of $\varphi\in\KG$ such that
$V\subseteq\ker(\varphi)$.
\smallskip

\begin{proposition}\label{prop:dual_RO}
  For $T\in\ROGG(G,<)$, the operator $T^!:=\idd{\KG}-T^*$ is a reduction
  operator on $\KG$. Moreover, we have $\nff(T^!)=\redd(T)$,
  $\redd(T^!)=\nff(T)$ and $\ker(T^!)=\ker(T)^{\bot}$.
\end{proposition}

\begin{proof}
  The operator $T$ being a projector, $T^*$ and $T^!$ are projectors. For
  every $g\in G$, $T^*(g)$ is equal to $\sum\dual{g}{T(g')}g'$, so that
  $T^!(g)=g-\sum\dual{g}{T(g')}g'$. If $g\ordop g'$, that is, $g'<g$, $g$
  does not belong to $\supp(T(g'))$, that is, $\dual{g}{T(g')}=0$. Hence,
  $T^!(g)=g-\sum\dual{g}{T(g')}g'$, where the sum is taken over all $g'$'s
  such that $g'\ordopref g$. If $g\in\red(T)$, then $\dual{g}{T(g')}=0$,
  for every $g'\ordopref g$, so that $T^!(g)=g$. If $g\in\nf(T)$, then
  $\dual{g}{T(g)}=1$, so that $T^!(g)=-\sum\dual{g}{T(g')}g'$, the sum is
  taken over $g'$'s such that $g'\ordop g$. Hence, $\red(T)$ and $\nf(T)$
  are included in $\nf(T^!)$ and $\red(T^!)$, respectively, and by using
  that $\red(T)\cup\nf(T)=G$, these inclusions are equalities. Moreover,
  denoting by $\delta$ the metric defined such as in \eqref{equ:metric},
  for every $v\in\K G$, we have $\delta(T^!(v),0)\leq\delta(v,0)$, so
  that $T^!$ is continuous at $0$, hence continuous. Hence, $T^!$ is a
  reduction operator on $\KG$. It remains to show the relation on
  kernels. For $\varphi\in\ker(T^!)$ and $g\in\red(T)$, we have
  $\dual{\varphi}{g-T(g)}=\dual{T^!(\varphi)}{g}=0$, so that $\varphi$
  vanishes over the set $\{g-T(g)\mid g\in\red(T)\}$. This set forms a
  basis of $\ker(T)$, so that $\varphi\in\ker(T)^{\bot}$. Conversely let
  $\varphi\in\ker(T)^{\bot}$ and $v\in\K G$. We have
  $\dual{T^!(\varphi)}{v}=\dual{\varphi}{v-T(v)}=0$, so that
  $\varphi\in\ker(T^!)$.
\end{proof}
\smallskip

From Proposition~\ref{prop:dual_RO}, we have a map
\begin{equation}\label{equ:duality_for_RO}
  \ROG(G,<)\to\ROG(\widehat{G},\ordop),\ T\mapsto T^!.
\end{equation}
Moreover, if $V$ and $W$ are subspaces of $\K G$, then $V\subseteq W$
implies $W^{\bot}\subseteq V^{\bot}$. Hence, the equality
$\ker(T^!)=\ker(T)^{\bot}$, for $T\in\ROG(G,<)$, implies that the map
\eqref{equ:duality_for_RO} is strictly decreasing. 
\medskip

We finish this Section by relating the duality for reduction operators to
representations of series. 
\smallskip

\begin{definition}
  Let $S\in\KG$ be a formal series.
  \begin{itemize}
  \item The \emph{representations category} of $S$ is the category
    defined as follows:
  \begin{itemize}
  \item objects are triples $(V,\alpha,\varphi)$, where $V$ is a vector
    space, $\alpha:V\to\K G$ is a linear map, and $\varphi$ a linear form
    on $V$ such that $S=\varphi\circ\alpha$;
  \item a morphism between two representations $(V,\alpha,\varphi)$ and
    $(V',\alpha',\varphi')$ is a linear map $\phi:V\to V'$ such that
    $\phi\circ\alpha=\alpha'$ and $\varphi'\circ\phi=\varphi$.
  \end{itemize}
  \item A representation is said to be \emph{surjective} if $\alpha$ is
    surjective.
  \item A \emph{representation by operator} is a representation
    $(\K\nf(T),T,S_{\mid\K\nf(T)})$, where $S_{\mid\K\nf(T)}$ is the
    restriction of $S$ to $\K\nf(T)$. In this case, we say that $S$ is
    \emph{represented} by $T$.
  \end{itemize}
\end{definition}
\smallskip
\noindent
We point out that two representations are isomorphic if and only if there
exists a morphism of representations between them which is an isomorphism
as a linear map and that a representation by operator is surjective.
\medskip

The following proposition means that being a surjective representation is
a duality condition.

\begin{proposition}\label{prop:surjective_representations}
  A surjective representation of $S\in\KG$ is isomorphic to a
  representation by operator of $S$. Moreover, S is represented by a
  reduction operator T if and only if $S\in\ker(T^!)$.
\end{proposition}

\begin{proof}
  Assume that $(V,\alpha,\varphi)$ is a surjective representation of $S$,
  so that $V$ is the quotient of $\K G$ by $\ker(\alpha)$. Let $T$ be the
  reduction operator such that $\ker(T)=\ker(\alpha)$, so that there is
  an isomorphism $\phi:V\to\K\nf(T),\ \alpha(u)\mapsto T(u)$, with
  inverse $\phi^{-1}(u)=\alpha(u)$. In order to show that $\phi$ is a
  morphism of representations, we only have to show that
  $\varphi(v)=S(\phi(v))$, for every $v\in V$. Given $u\in\K G$ such that
  $\alpha(u)=v$, we have $\varphi(v)=S(u)$ and
  $S(\phi(v))=S(\phi(\alpha(u)))=S(T(u))=S(u)$. Hence, $\phi$ is an
  isomorphism of representations. The second assertion of the proposition
  is due to the fact that the relation $S=S\circ T$,  means $S=T^*(S)$,
  that is, $S\in\ker(T^!)$.
\end{proof}
\medskip

Finally, we classify series represented by a single reduction operator.

\begin{proposition}\label{prop:representations_and_duality}
  Let T be a reduction operator on $\K G$. Then, $T^*$ induces an
  isomorphism between $\widehat{\K\nff(T)}$ and series represented by T.
\end{proposition}

\begin{proof}
  From Proposition~\ref{prop:surjective_representations}, $S\in\KG$ is
  represented by $T$ if and only if $S\in\ker(T^!)$. Moreover the set of
  $g-T^!(g)$, $g\in\red(T^!)=\nf(T)$, forms a total basis of $\ker(T^!)$.
  Hence, there exists a unique $S'\in\widehat{\K\nf(T)}$ such that
  $S=S'-T^!(S')=T^*(S')$. The map $S\mapsto S'$ is the inverse of $T^*$.
\end{proof}

\bigskip

\newpage

\begin{center}
  \bf\large 4. APPLICATIONS TO FORMAL POWER SERIES
  \addcontentsline{toc}{section}{4. Applications to formal power series}
  \setcounter{section}{4}
\end{center}

\smallskip

In this Section, we apply the theory of topological reduction operators
presented in the previous Sections to formal power series. For that, we
recall some notions introduced in Example~\ref{ex:pol_and_formal_series},
and we fix some conventions and notations.
\smallskip

In the following two Section, we fix a set $X:=\{x_1,\cdots,x_n\}$ of
indeterminates. We denote by $\cPol$ and $\ncPol$ the commutative and
noncommutative polynomial algebras over $X$, respectively. As vector
spaces, these algebras have a basis composed of commutative and
noncommutative monomials, respectively, the noncommutative being
identified to words. A \emph{monomial order} is a well-order on
monomials, compatible with multiplication. A (non)commutative
\emph{\Gr\ basis} of a (two-sided) ideal $I$ of $\cPol$ or $\ncPol$, is a
generating subset $R$ of $I$ such that $\lm(R)$ is a generating subset of
the monomial ideal $\lm(I)$. In other words, $R$ is a (non)commutative
\Gr\ basis if and only if for every $f\in I$, there exists $g\in R$ such
that $\lm(g)$ divides $\lm(f)$. We recall that this is equivalent to the
fact that the polynomial reduction induced by $R$ is a confluent
rewriting relation. In this case, the irreducible monomials for the
polynomial reduction form a linear basis of the quotient algebra $A/I$,
where $A=\cPol$ or $\ncPol$.
\smallskip

Denote by $A$ the algebra $\cPol$ or $\ncPol$. For a subset $S$ of $A$,
we denote by $I(S)$ the two-sided ideal generated by $S$. We equip $A$
with the $I(X)$-\emph{adic} topology, that is, the topology induced by
the metric $\delta(f,g)=1/2^n$, where $n$ is the smallest degree of a
monomial occurring in the decomposition of $f-g$. The sets $\cSeries$ and
$\ncSeries$ of commutative and noncommutative formal power series,
respectively, are the completions of the corresponding algebras. Note
that the support of a formal power series $\sum\alpha_mm$, where the sum
is taken over monomials, is the set of monomials $m$ such that $\alpha_m$
is different from $0$.

\bigskip

\noindent
{\large\textbf{4.1. Topological confluence and standard bases}}
\addcontentsline{toc}{subsection}{4.1. Topological confluence and standard bases}
\setcounter{subsection}{1}
\setcounter{theorem}{0}
\bigskip

Throughout this Section, we only deal with commutative formal power
series. Let $<$ be a monomial order on commutative monomials, which is
assumed to be compatible with degrees: $\deg(m)<\deg(m')$ implies $m<m'$.
Let $\ordred:=\ordop$ be the opposite order of $<$, so that the leading
monomial of a formal power series is the smallest element of its support
with respect to $<$. 

\smallskip

\begin{definition}
  Let $I$ be an ideal of $\cSeries$. A \emph{standard basis} of $I$ is a
  generating set $R$ of $I$ such that $\lm(R)$ generates the monomial
  ideal $\lm(I)$.
\end{definition}
\medskip

Let us consider the map $d$ which maps every monomial $m$ to
$1/2^{\deg(m)}$, so that $d$ is an elimination map and the metric induced
by $d$ is precisely the metric $\delta$ of the $I(X)$-adic topology. Our
purpose is to relate standard bases for power series ring ideals to the
confluence property of reduction operators. For that, for any
$f\in\cSeries$, we denote by $T(f)$ the reduction operator whose kernel
is the closed ideal generated by $f$: $T(f):=\ker^{-1}(\overline{I(f)})$.
Explicitly, for every monomial $m$, $T(f)$ is defined by the following
recursive formulas:
\begin{itemize}
\item if $m$ is not divisible by $\lm(f)$, then $T(f)(m)=m$,
\item if $m=\lm(f)m'$, then $T(f)(m)=1/\lc(f)
  \Big(T(f)(m'(\lc(f)\lm(f)-f))\Big)$.
\end{itemize}
In particular, $\red(T(f))$ is the monomial ideal spanned by $\lm(f)$.
Finally, for a subset $R\subseteq\cSeries$, we denote by
$F(R):=\{T(f)\mid f\in R\}$.
\smallskip

\begin{proposition}\label{prop:lattice_car_of_GB}
  A subset $R$ of $\cSeries$ is a standard basis of the ideal it
  generates if and only if $F(R)$ is a confluent set of reduction
  operators.
\end{proposition}

\begin{proof}
  We denote by $I(R)$ the ideal of $\K[[X]]$ generated by $R$. By
  definition of $\wedge F(R)$, its kernel is equal to $\overline{I(R)}$.
  Hence, $\red(\wedge F(R))$ is equal to $\lm(\overline{I(R)})$, and
  from~\eqref{equ:leading_monomials}, the latter is equal to
  $\lm(I(R))$. Moreover, $\red(F(R))$ is the union of the sets
  $\lm(T(f))$, $f\in R$, that is, it is the monomial ideal spanned by
  $\lm(R)$. Hence the statement of the proposition is due to the
  following sequence of equivalences: $F(R)$ is confluent if and only if
  $\red(\wedge F(R))$ is equal to $\red(F(R))$, that is, if and only if
  $\lm(I(R))$ is the monomial ideal spanned by $\lm(R)$, that is, if and
  only if $R$ is a standard basis of $I(R)$.
\end{proof}
\smallskip

As for \Gr\ bases and $S$-polynomials, there is a criterion in terms of
\emph{S-series} for a generating set of an ideal of a power series ring
to be a standard basis~\cite[Theorem 4.1]{MR1075338}. From
Theorem~\ref{thm:lattice_confluence} and
Proposition~\ref{prop:lattice_car_of_GB}, we obtain the following
formulation of this criterion in terms of the $\delta$-confluence
property.
\smallskip

\begin{theorem}\label{thm:standard_bases_confluence}
  A subset $R$ of $\cSeries$ is a standard basis of the ideal it
  generates if and only if the rewriting relation $\to_{F(R)}$ is
  $\delta$-confluent.
\end{theorem}

\medskip

\begin{example}
  Consider the example of the introduction: $X:=\{x,y,z\}$, $<$ is the
  deglex order induced by $x>y>z$ and $R:=\{z-y,z-x,y-y^2,x-x^2\}$. Then,
  $\lm(R)$ is equal to $\{x,y,z\}$, and constant coefficients of elements
  of the power series ideal generated by $R$ are equal to $0$, so that
  $R$ is a standard basis of this ideal. Hence, the rewriting relation
  $\to_{F(R)}$ is $\delta$-confluent. However, it is not confluent as
  illustrated by the following diagram, where, for simplicity, we remove
  the subscript $F(R)$:
  \[\begin{tikzcd}
  &
  x\arrow[r] & x^2\arrow[r] & \cdots\arrow[r] & x^{2n}\arrow[Rightarrow, rd,
    bend left] & 
  \\
  z\arrow[ru, bend left]\arrow[rd, bend right] & & & & & 0
  \\
  &
  y\arrow[r] & y^2\arrow[r] & \cdots\arrow[r] & y^{2n}\arrow[Rightarrow,
    ru, bend right] & 
  \end{tikzcd}\]
\end{example}

\bigskip

\noindent
{\large\textbf{4.2. Duality and syntactic algebras}}
\addcontentsline{toc}{subsection}{4.2. Duality and syntactic algebras}
\setcounter{subsection}{2}
\setcounter{theorem}{0}
\bigskip

Throughout out this Section, we only deal with noncommutative objects:
formal power series, algebras, \Gr\ bases $\hdots$ Hence, we omit the
adjective noncommutative. Moreover, we only deal with two-sided ideals,
so that we also omit the adjective two-sided.
\medskip

Our purpose is to relate duality for reduction operators to syntactic
algebras. Let $S\in\K\Span{\Span{X}}$ be a formal power series. The
\emph{syntactic ideal} of $S$, written $I_S$, is the greatest ideal
included in $\ker(S)$. The \emph{syntactic algebra} of $S$ is the
quotient algebra $A_S:=\K\Span{X}/I_S$. The series $S$ is said to be
\emph{rational} if $A_S$ is finite-dimensional as a vector space.
Moreover, an algebra is said to be \emph{syntactic} if it is the
syntactic algebra of a formal power series. Let us illustrate this notion
with an example coming from~\cite{petitot1992algebre}: $X=\{x_0,x_1\}$,
and 

\begin{equation}\label{equ:words-series}
  S:=\sum_{w\in X^*}\val(w)w,
\end{equation}
\noindent
where $\val(w)\in\N$ is the integer whose binary expression is equal to
$w$. From~\cite{petitot1992algebre}, this series is rational. We propose
another proof of this result by providing a \Gr\ basis of the syntactic
ideal of $S$.
\smallskip

\begin{proposition}\label{prop:syntactic_ideal_words}
  The syntactical ideal of \eqref{equ:words-series} admits the following
  \Gr\ basis:
  \[R:=\left\{
  \begin{tabular}{l l}
    $f_1:=x_0x_0-3x_0+2$ & $f_2:=x_0x_1-x_1-2x_0+2$\\[0.2cm]
    $f_3:=x_1x_0-2x_1-x_0+2$ & $f_4:=x_1x_1-3x_1+2$
  \end{tabular}
  \right\}\subset\K\Span{X}.\]
  In particular, \eqref{equ:words-series} is rational.
\end{proposition}

\begin{proof}
  The ideal generated by $f_1$ is included in $\ker(S)$, since for every
  words $w$ and $w'$, we have
  \[\begin{split}
  \dual{S}{wf_1w'}&= \dual{S}{wx_0x_0w'}-3\dual{S}{wx_0w'}+2\dual{S}{ww'}\\
  &=2^{\len{x_0x_0w'}}\val(w)+\dual{S}{x_0x_0w'}-3\Big(
  2^{\len{x_0w'}}\val(w)+\dual{S}{x_0w'}\Big)
  +2\Big((2^{\len{w'}}\val(w)+\dual{S}{w'}\Big)\\
  &= 2^{\len{w'}}\Big(4-6+2).\val(w)+2^{\len{w'}}\dual{S}{f_1}+
  \dual{S}{w'}\Big(1-3+2\Big)\\
  &=0.
  \end{split}\]
  By using analogous arguments, we show that ideals generated by other
  $f_i$'s are also included in $\ker(S)$. We easily show that for the
  deglex order induced by $x_0<x_1$, all the $S$-polynomials of $f_i$'s
  reduce into zero, so that $R$ forms a \Gr\ basis of the ideal $I$ it
  generates. Hence, the algebra $A:=\K\Span{X}/I$ is 3-dimensional, with
  a basis composed of $1,x_0$ and $x_1$. Moreover, by computing
  codimensions, we check that $\ker(S)$ is equal to $I\oplus\K\{1,x_0\}$,
  so that $I$ is the syntactic ideal of $S$, and $A$ is its syntactic
  algebra.
\end{proof}

We associate to an algebra $A=\K\Span{X}/I$, the reduction operator
$T:=\ker^{-1}(I)$ on $\ncPol$ with kernel~$I$. This operator is computed
as follows: if $R$ is a \Gr\ basis of $I$, then $T(f)$ is the unique
normal form of $f\in\K\Span{X}$ through polynomial reduction. In
particular, a series is represented by $A$ if and only if it is
represented by~$T$.
\smallskip

\begin{theorem}\label{thm:syntactic_algebra_duality}
  Let $I\subseteq\K\Span{X}$ be an ideal and let T be the reduction
  operator with kernel I. Then, the algebra $\ncPol/I$ is syntactic if
  and only if there exists $S'\in\nSeriess$ such that I is the greatest
  ideal included in  $I\oplus\ker(S')$.
\end{theorem}

\begin{proof}
  First, observe that the sum of the statement of the Theorem is direct
  since for every subspace $V\subseteq\nPol$, $I+V$ is direct, indeed,
  the leading monomial of a nonzero $f\in I$ does not belong to $\nf(T)$.
  Moreover, from Proposition~\ref{prop:representations_and_duality}, $S$
  is represented by $A$ if and only if there exists $S'\in\nSeries$ such
  that $S=T^*(S')$. Hence, for every $g\in\ncPol$, we have
  $\dual{S}{g}=\dual{S'}{T(g)}$, so that
  \begin{equation}\label{equ:nc_series_duality}
    \dual{S}{g}=\dual{S}{(g-T(g))+T(g)}=\dual{S'}{T(g)}.
  \end{equation}
  The element $g-T(g)$ belongs to $I$, which is included in $\ker(S)$.
  Thus, from \eqref{equ:nc_series_duality}, $\ker(S)=I\oplus\ker(S')$.
  The statement of the theorem follows since $S$ is represented by its
  syntactic algebra~\cite{MR593604}.
\end{proof}

\medskip

\begin{example}
  \begin{enumerate}
  \item Consider the series $S$ as in \eqref{equ:words-series} and let
    $T$ be the reduction operator with kernel the syntactic ideal $I_S$.
    From Proposition~\ref{prop:syntactic_ideal_words}, $\nf(T)$ is equal
    to $\{1,x_0,x_1\}$, so that $S=T^*(S')$, $S'\in\K\nf(T)$. By
    evaluating $S$ at $1,\ x_0$ and $x_1$, we get $S'=x_1$. We easily
    check that $I_S$ is the greatest ideal included in
    $I_S\oplus\ker(S')=I_S\oplus\K\{1,x_0\}$.
  \item Consider the algebra $A:=\K\Span{X}/I$, where $X:=\{x,y\}$
    and $I$  is the ideal generated by $2$-letter words.
    In~\cite{MR593604}, it is proven that this algebra is not syntactic.
    Here, we propose another proof, based on duality. The reduction
    operator $T$ with kernel $I$ maps every word of length at least $2$
    to~$0$. Let $S'=\alpha+\beta x+\gamma y\in\K\nf(T)$. If
    $\beta=\gamma=0$ , then $I\oplus\ker(S')$ is equal to
    $I\oplus\K\{x,y\}$, and if not, it contains
    $I\oplus\K\{\gamma x-\beta y\}$. Since both $I\oplus\K\{x,y\}$ and
    $I\oplus\K\{\gamma x-\beta y\}$ are ideals, $I\oplus\ker(S')$
    contains an ideal strictly greater than $I$. Hence, the criterion of
    Theorem~\ref{thm:syntactic_algebra_duality} does not hold, that is,
    $A$ is not syntactic.
  \end{enumerate}
\end{example}

\bibliography{Biblio}

\end{document}